\documentclass[10pt,a4paper]{amsart}
\usepackage{amssymb,amsmath,amsthm,eucal}
\usepackage{amscd}
\usepackage[latin1]{inputenc}
\usepackage[all]{xy}
\usepackage{color}

\usepackage[latin1]{inputenc}

\usepackage{mathpazo}
\usepackage[scaled=.95]{helvet}
\usepackage{courier}

\usepackage[bookmarks,colorlinks,pagebackref]{hyperref} 

\usepackage{microtype}

\numberwithin{equation}{section}
\theoremstyle{plain} 
\newtheorem{proposition}{Proposition}[section]  

\newtheorem{corollary}[proposition]{Corollary} 
\newtheorem{theorem}[proposition]{Theorem} 

\theoremstyle{definition} 
\newtheorem{definition}[proposition]{Definition}

\newtheorem{remark}[proposition]{Remark} 

\newtheorem{example}[proposition]{Example} 

\newtheorem{notation}[proposition]{Notation}

{}
{}
{}
{}
{}
{}

\newcommand\Tor{\operatorname{Tor}}
\newcommand\Hom{\operatorname{Hom}}
\newcommand\RHom{\operatorname{R Hom}}

\newcommand\Rad{\operatorname{Rad}}

\newcommand{\xx}{\underline x}

\newcommand{\ff}{\underline f}

\newcommand\Spec{\operatorname{Spec}}

\newcommand{\qism}{\stackrel{\sim}{\longrightarrow}}

{}

\author[Peter Schenzel]{Peter Schenzel}

\title[proregular sequences]{{About proregular sequences and an application to prisms}}

\address{Martin-Luther-Universit\"at Halle-Wittenberg,
Institut f\"ur Informatik, D --- 06 099 Halle (Saale), Germany}
\email{schenzel@informatik.uni-halle.de}

\subjclass[2020]
{Primary: 13C12; Secondary: 13C11, 13D07}
\keywords{torsion, injective module,  non-Noetherian commutative ring, proregular sequences}

\begin{document}
\begin{abstract}
	Let $\xx = x_1,\ldots,x_k$ denote an ordered sequence of elements of a commutative ring $R$. Let $M$ 
	be an $R$-module. We recall the two notions that $\xx$ is $M$-proregular given by Greenlees and May 
	(see \cite{GM}) and Lipman (see \cite{lip}) and show that both notions are equivalent. As a main result 
	we prove a cohomological characterization for $\xx$ to be $M$-proregular in terms of 
	\v{C}ech homology. This implies also that $\xx$ is $M$-weakly proregular if it is $M$-proregular. 
	A local-global principle for proregularity and weakly proregularity is proved. This is used for a result about prisms as introduced by Bhatt and Scholze (see \cite{BSc}).
\end{abstract}

\maketitle

\section*{Introduction}
Let $R$ denote a commutative ring and let $M$ be an $R$-modue. Let $\xx = x_1,\ldots,x_k$ denote an ordered sequence  of elements of $R$. In their paper Greenlees and May (see \cite{GM}) defined $\xx$ to be $M$-proregular if 
for all $i = 1,\ldots,k$ and an integer $n \geq 1$ there is an integer $m \geq n$ such that the multiplication map 
\[
(x_1,\ldots,x_{i-1})^mM :_M x_i^m/(x_1,\ldots,x_{i-1})^mM  \stackrel{x_i^{m-n}}{\longrightarrow} 
(x_1,\ldots,x_{i-1})^nM :_M x_i^n/(x_1,\ldots,x_{i-1})^nM 
\]
is zero.  In their paper (see \cite{lip}) Lipman et all called $\xx$ to be $M$-proregular if for all $i = 1,\ldots,k$ and 
an integer $n \geq 1$ there is an integer $m \geq n$ such that the multiplication map 
\[
(x_1^m,\ldots,x_{i-1}^m)M :_M x_i^m/(x_1^m,\ldots,x_{i-1}^m)M  \stackrel{x_i^{m-n}}{\longrightarrow} 
(x_1^n,\ldots,x_{i-1}^n)M :_M x_i^n/(x_1^n,\ldots,x_{i-1}^n)M  
\]
is zero (see also the Definition \ref{def-1}). The advantage of the second notion of proregularity  is its relation to Koszul complexes. 
Moreover, both definitions 
are equivalent  (see \ref{prop-1}).  The notion of proregular sequences is needed 
by Greenlees and May for their study of the left derived functors of the completion. This was continued by Lipman et all (see \cite{lip} and \cite{lip2}). Moreover, it turned out that the notion that $\xx$
 is weakly proregular (see \ref{def-2}) is more appropriate  for the study of \v{C}ech homology and cohomology. A first systematic study of weakly proregular sequences is done in \cite{S}.  See also the monograph \cite{SpSa} for a more systematic investigation  and its relation to  completions. As a first main result we prove a 
 homological characterization of $M$-proregular sequences.
 
 \begin{theorem} \label{thm-0}
 	Let $\xx = x_1,\ldots,x_k$ denote an ordered sequence of a commutative ring $R$. For an $R$-module $M$ the following conditions are equivalent.
 	\begin{itemize}
 		\item[(i)]  The sequence $\xx$ is $M$-proregular. 
 		\item[(ii)] $\check{H}^1_{x_i}(\Gamma_{\xx_{i-1}}(\Hom_R(M,I) ) = 0$ for $i = 1,\ldots,k$ and any injective $R$-module $I$. 
 	\end{itemize}
 \end{theorem}
 
 For the proof see \ref{thm-1}. Here $\check{H}^1_{x_i}(\cdot)$ denotes the \v{C}ech homology with respect to $x_i$ (see \ref{not-1}); and 
 $\Gamma_{\xx_{i-1}}(\cdot)$ is the torsion functor with respect to the ideal generated by 
 $\xx_{i-1} = x_1,\ldots,x_{i-1}$.  As an application there is a behavior similar to that of a regular sequence.  For a 
 sequence $\xx = x_1,\ldots,x_k$ and any $k$-tupel $(n_1,\ldots,n_k) \in (\mathbb{N}_+)^k$ we write 
 $\xx^{(\underline{n}) } = x_1^{n_1}, \ldots,x_k^{n_k}$ . 
 
\begin{corollary} \label{cor-0}
	With the notation of \ref{thm-0} the following conditions are equivalent.
		\begin{itemize}
		\item[(i))] The sequence $\xx$ is $M$-proregular.
		\item[(ii)] There is an $\underline{n} \in (\mathbb{N}_+)^k$ such that the sequence 
		$\xx^{(\underline{n})}$ is $M$-proregular.
		\item[(iii)] The sequence $\xx^{(\underline{n})}$ is $M$-proregular for all $\underline{n} \in (\mathbb{N}_+)^k$.
	\end{itemize}
\end{corollary}
 
If the sequence $\xx$ is $M$-regular it is also $M$-proregular. For a Noetherian $R$-module $M$ 
any sequence is $M$-proregular. Therefore, the notion of proregularity is important in the non-Noetherian 
situation. An application of this is the following result inspired by the work 
of Bhatt and Scholze (see \cite{BSc}).  For the notation we refer to Chapter 4.

\begin{corollary} \label{cor}
	Let $(R, \mathcal{I})$ denote a prism. Suppose that $\mathcal{I}$ is of bounded $p$-torsion.  Then it follows.
	\begin{itemize}
		\item[(a)]  $(\mathcal{I}, p)$ is proregular, i.e. for an integer $n$ there is an $m \geq n$ 
		such that $\mathcal{I}^m :_R p^m /\mathcal{I}^m \stackrel{p^{m-n}}{\longrightarrow} 
		\mathcal{I}^n :_R p^n /\mathcal{I}^n$ is the zero map. 
		\item[(b)] $\Gamma_{\mathcal{I}} (I)/\Gamma_{(\mathcal{I},p)}(I)$ is $p$-divisible for any 
		injective $R$-module $I$, i.e. $\Gamma_{\mathcal{I}} (I) = p \Gamma_{\mathcal{I}} (I) + 
		\Gamma_{(\mathcal{I},p)}(I)$.
	\end{itemize}
Moreover, the conditions (a) and (b) are equivalent. 
\end{corollary}

In Section 1 we summarize the notation and prove some basic results. Section 2 is devoted the homological approach of proregular sequences and their consequences. A few more results about $M$-weakly proregular 
sequences are shown in Section 3. This provides a relative version of some results of \cite{SpSa}.  In Section 4 we investigate the local-global behavior of 
proregular and weakly proregular sequences and prove the Corollary \ref{cor}. For results of Commutative 
Algebra we follow Matsumura's book \cite{mats}. For homological preliminaries and definitions about 
\v{C}ech homology and cohomology we refer to the monograph \cite{SpSa}.

\section{Definitions and Preliminaries}
In the following let $R$ denote a commutative ring. Let $\xx = x_1,\ldots, x_k$ denote an ordered 
sequence of elements of $R$. For a positive integer $n$ we put $\xx^{(n) } = x_1^n,\ldots,x_k^n$. Let $M$ denote 
an $R$-module.  Note that 
\[
 \xx^{nk}M \subseteq \xx^{(n)} M \subseteq \xx^{n}M \; \mbox{  for all  } \; n \geq 1
\] 
and therefore $\Rad \xx R = \Rad \xx^{(n)}R$ for all $n \geq 1$.  The  $R$-module $M$ is of bounded 
$x$-torsion for an element $x \in R$ if the increasing sequence $\{0:_M x^n\}_{n \geq 1}$ 
stabilizes, i.e. there is an integer $c$ such that $0:_M x^n = 0:_M x^c$ for all $n \geq c$. Note that 
$M$ is of bounded $x$-torsion for an $M$-regular element $x \in R$. If $N$ is a submodule of $M$ and 
$M$ is of bounded $x$-torsion, then $N$ is of bounded $x$-torsion too. Moreover, if $M$ is a 
Noetherian module, then $M$ is of bounded torsion for any element. Let $m \geq n$ denote two 
positive integers. The multiplication map by $x^{m-n}$ on $M$ induces a map 
\[
0:_M x^m \stackrel{x^{m-n}}{\longrightarrow} 0:_M x^n \; \mbox{ for all } \; m \geq n.
\]
Therefore $M$ is of bounded $x$-torsion if and only if the previous inverse system is pro-zero. This 
leads Greenlees and May to the definition of an $M$-proregular sequence (see \cite[Definition 1.8]{GM}). 
We follow here the notion of proregularity as given in \cite{lip} and show that it is equivalent to the notion 
of Greenlees and May (see \cite{GM}). 

\begin{definition} \label{def-1}
	Let $M$ denote an $R$-module and let $\xx = x_1,\ldots, x_k$ denote a sequence of elements 
	of $R$. Then $\xx$ is called M-proregular if for $i = 1,\ldots,k$ and any positive integer $n$ 
	there is an integer $m \geq n$ 
	such that the multiplication map 
	\[
	(x_1^m,\ldots,x_{i-1}^m)M :_M x_i^m/(x_1^m,\ldots,x_{i-1}^m)M  \stackrel{x_i^{m-n}}{\longrightarrow} 
	(x_1^n,\ldots,x_{i-1}^n)M :_M x_i^n/(x_1^n,\ldots,x_{i-1}^n)M  \eqno (1)
	\]
	is zero. This is equivalent to saying that for $i = 1,\ldots,k$ and any positive integer $n$ there is an integer 
	$m \geq n$ such that 
	\[
		(x_1^m,\ldots,x_{i-1}^m)M :_M x_i^m \subseteq 	(x_1^n,\ldots,x_{i-1}^n)M  :_M x_i^{m-n}. \eqno (2)
	\]
	Note that an element $x\in R$ is $M$-proregular if and only if $M$ is of bounded $x$-torsion.
\end{definition}

A first result about the behavior of $M$-proregular sequences is the following.

\begin{proposition} \label{prop-1}
	Let $\xx = x_1,\ldots,x_k$ denote an ordered sequence of elements of $R$. Let $M$ denote 
	an $R$-module. Then the following conditions are equivalent.
	\begin{itemize}
		\item[(i)] The sequence $\xx$ is $M$-proregular.
		\item[(ii)] For $i = 1,\ldots,k$ and any positive integer $n$ there is an integer $m \geq n$ such that 
		\[
		(x_1,\ldots,x_{i-1})^mM :_M x_i^m \subseteq 	(x_1,\ldots,x_{i-1})^nM:_M x_i^{m-n}. 
		\]
		\item[(iii)] For $i = 1,\ldots,k$ and any  integer $n$ there is an integer $m \geq n$ such that 
		the multiplication map 
		\[
			(x_1,\ldots,x_{i-1})^mM :_M x_i^m/(x_1,\ldots,x_{i-1})^mM  \stackrel{x_i^{m-n}}{\longrightarrow} 
		(x_1,\ldots,x_{i-1})^nM :_M x_i^n/(x_1,\ldots,x_{i-1})^nM 
		\]
		is zero.
	\end{itemize}
\end{proposition}

\begin{proof}
	We fix $l$ and put $\xx = x_1,\ldots,x_{l-1}$ and $y = x_l$. 
	The equivalence of (ii) and (iii) is trivially true. Then we prove the implication (i) $\Longrightarrow$ (iii). 
	By the assumption and the definition \ref{def-1} for a given $n$ there is an $m \geq n$ such that 
	$\xx^{(m)} M :_M y ^m \subseteq \xx^{(n)} M :_M y^{m-n}  \subseteq \xx^nM:_M y^{m-n}$. Because of 
	$\xx^{ml}M :_M y^{ml} \subseteq \xx^{(m)}M :_M y^{ml} $ this implies that 
	\[
	\xx^{ml}M :_M y^{ml} \subseteq \xx^{(m)}M :_M y^{ml}  \subseteq \xx^{(n)}M :_M y^{ml-n}  
	\subseteq \xx^n M:_M y^{ml-n}
	\]
	as required. In order to prove (ii) $\Longrightarrow$ (i) fix $n$ and choose an integer $m \geq nl$ such 
	that $\xx^mM:_M y^m \subseteq  \xx^{nl}:_M y^{m-nl}$. Then we get the inclusions 
	\[
	\xx^{(m)}M:_M y^m \subseteq  \xx^{m} M:_M y^m \subseteq \xx^{nl} M:_M y^{m-nl} 
	\subseteq \xx^{(n)}M:_M y^{m-nl}  
	\subseteq \xx^{(n)}M:_M y^{m-n}
	\]
	which finishes the proof.
\end{proof}

The following example by J. Lipman (see \cite{lip2}) shows that a proregular sequence 
is not permutable. 

\begin{example} \label{ex-1}
	Define $R = \prod_{n \geq 1} \mathbb{Z}/2^n\mathbb{Z}$ . Let $x = (2+2^n)_{n \geq 1}$ and 
	$1 = (1+2^n)_{n\geq 1}$ two elements of $R$. Then the sequence $\{1,x\}$ is $R$-regular, 
	in particular $R$-proregular. Moreover, $R$ is not of bounded $x$-torsion. Therefore 
	$\{x,1\}$ is not a proregular sequence. 
\end{example}

Clearly an $M$-regular sequence $\xx$ is also $M$-proregular. An interesting extension is 
the following result.

\begin{proposition} \label{prop-2}
	Let $M$ be an $R$-module. Let $\xx = x_1,\ldots,x_k$ denote an $M$-regular sequence.
	 Suppose that $M/\xx M$ is of bounded $y$-torsion for an element $y \in R$. Then 
	 $\xx,y = x_1,\ldots,x_k,y$ is an $M$-proregular sequence. 
\end{proposition}

\begin{proof}
	Now  $x_i$ is regular on $M/(x_1,\ldots,x_{i-1})^nM$ for all $i =1,\ldots,k$ and all $n \geq 1$ 
	since $\xx = x_1,\ldots,x_k$ is an $M$-regular sequence (see e.g. \cite[Theorem 16.2]{mats}). 
	That is, condition (ii) in \ref{prop-1} is satisfied for all $i = 1,\ldots, k$. In order to finish the 
	proof we have to show that $M/\xx^nM$ is of bounded $y$-torsion for all $n \geq 1$. Namely, 
	if $\xx^nM:_M y^c = \xx^n M:_M y^d$ for all $d \geq c$ choose $m = n +c$ and therefore 
	$\xx^m M:_M y^m \subseteq \xx^nM:_M y^m = \xx^nM:_M y^{m-n}$. It remains to show the 
	previous claim. Since $\xx$ is an $M$-regular sequence $\xx^nM/\xx^{n+1} M\cong \oplus_{b_n}M/\xx M$ 
	with $b_n = \binom{k+n-1}{n} $ (see e.g. \cite{mats}). Then the short exact sequence 
	\[
	0 \to \xx^nM/\xx^{n+1} M \to M/\xx^{n+1} M \to  M/\xx^nM \to 0
	\]
	implies by applying $\Gamma_y$ the exact sequence $0 \to \Gamma_y( \xx^nM/\xx^{n+1} M) \to 
	\Gamma_y(M/\xx^{n+1} M ) \to \Gamma_y(M/\xx^nM)$. This 
	proves -- by induction on $n$ -- that $M/\xx^nM$ is of bounded $y$-torsion for all $n \geq 1$. Note that a submodule 
	of $\Gamma_y(M/\xx^nM)$ is of bounded $y$-torsion too.
\end{proof}

For our further investigations we need a few more notation and definitions. 

\begin{notation} \label{not-1}
	(A) Let $\mathfrak{a} \subset R$ denote an ideal of $R$. For an $R$-module $X$ let 
	$\Gamma_{\mathfrak{a}}(X)  = \{x \in X| \mathfrak{a}^n x = 0 \mbox{  for some } n \geq 1 \}$ the 
	$\mathfrak{a}$-torsion submodule of $X$. Its left derived functors $H^i_{\mathfrak{a}}(X), i \geq 0,$ are the 
	local cohomology modules of $X$ with respect to $\mathfrak{a}$ (see e.g. \cite{BS} or \cite{SpSa}
	 for more details). \\
	 (B) Let $\xx = x_1,\ldots,x_k$ denote a sequence of elements of $R$. Let $\check{C}_{\xx}$ denote the 
	 \v{C}ech complex with respect to $R$. For an $R$-module $X$ we write $\check{C}_{\xx} (X) = 
	 \check{C}_{\xx} \otimes_R X$. We denote the cohomology of $\check{C}_{\xx} (X) $ by 
	 $\check{H}^i_{\xx}(X), i \geq 0$. \\
	 (C) Let $\xx = x_1,\ldots, x_k$ a sequence of elements and $\mathfrak{a}= \xx R$. Then it follows easily that 
	  $\Gamma_{\mathfrak{a}}(X) = \check{H}^0_{\xx}(X)$. The more general isomorphisms 
	 $H^i_{\mathfrak{a}}(X)  \cong \check{H}^i_{\xx}(X)$ for all $i \geq 0$ do not hold in general. 
	 They hold if and only if $\xx$ is a weakly proregular sequence (see \cite{SpSa} for more details 
	 about weakly proregular sequences).  \\
	 (D) For a sequence  of elements $\xx = x_1,\ldots,x_k$ and an $R$-module $M$ we use the 
	 Koszul complexes $K_{\bullet}(\xx;M)$ and $K^{\bullet}(\xx;M)$. We refer to \cite[5.2]{SpSa} for all the details we need. The Koszul homology and Koszul cohomology are denoted by $H_i(\xx;M)$ and $H^i(\xx;M)$ resp. for 
	 $i \in \mathbb{Z}$. 
\end{notation}

We continue with the a further definition. 

\begin{definition} \label{def-2} (see \cite[Section 7.3]{SpSa})
	Let $\xx = x_1,\ldots,x_k$ denote a sequence of elements of $R$. Let $M$ be an $R$-module. The sequence 
	$\xx$ is called $M$-weakly proregular if for all $i > 0$ and any positive integer $n$ there is an integer 
	$m \geq n$ such that the natural homomorphism 
	\[
	H_i(\xx^{(m)}; M)  \to 	H_i(\xx^{(n)}; M)
	\]
	is zero. A sequence $\xx$ is called weakly proregular if $\xx$ is $R$-weakly proregular. 
\end{definition}

The notion of weakly proregular sequences plays an essential r\^ole in respect to local cohomology 
and the left derived functors of completion (see \cite{lip}, \cite{lip2} and \cite{SpSa}). A characterization of 
$M$-weakly proregular sequence is the following. 

\begin{proposition} \label{prop-3}  {\rm (see  \cite[Proposition 5.3]{Sp})} 
    Let $\xx$ denote a sequence of elements of $R$. Let $M$ be an $R$-module. Then the following conditions are 
	equivalent.
	\begin{itemize}
		\item[(i)] $\xx$ is $M$-weakly proregular.
	\item[(ii)] The inverse  system $\{H_i(x^{(n)};M \otimes_R F)\}_{n \geq 1}$ is pro-zero for all $i > 0$ 
	and any flat $R$-module $F$.
	\item[(iii)] $\varinjlim H^i(\xx^{(n)};\Hom_R(M,I)) = 0$ for all $i > 0$ 
	and any injective $R$-module $I$.
	\item[(iv)] $\check{H}^i_{\xx}( \Hom_R(M,I)) = 0$ for all 
	$i > 0$ and any injective $R$-module $I$. 
	\end{itemize}
\end{proposition}

We shall see that an $M$-proregular sequence is also $M$-weakly proregular. The converse is not true 
as follows by \ref{ex-1} since the sequence $1,x$ is $R$-regular and in particular $R$-proregular. 

\section{A homological approach}
At first we will give an interpretation of the definition of an $M$-proregular sequence $\xx = x_1,\ldots,x_k$ 
in terms of Kozul complexes. For the basics about Koszul complexes we refer to \cite{SpSa}. We fix $l \in 
\{1,\ldots,k\}$ and put $\xx = x_1,\ldots,x_{l-1}$ and $y = x_l$. Then the natural map 
\[
\xx^{(m)}M :_M y^m/\xx^{(m)}M \stackrel{y^{m-n}}{\longrightarrow} \xx^{(n)}M :_M y^n/\xx^{(n)}M \; 
\mbox{  for all } \; m \geq n \eqno (3)
\]
coincides with the natural map 
\[
H_1(y^m; H_0(\xx^{(m)};M)) \stackrel{y^{m-n}}{\longrightarrow}H_1(y^n; H_0(\xx^{(n)};M)) \; 
\mbox{  for all } \; m \geq n. \eqno (4)
\]
induced by the natural map of Koszul complexes $K_{\bullet}(\xx^{(m)},y^m;M) \to 
K_{\bullet}(\xx^{(n)},y^n;M)$ (see \cite[Section 5.2]{SpSa} for some details). In the following we specify $\xx_i= 
x_1,\ldots,x_i$ for $i= 1,\ldots, k-1$. 

\begin{theorem} \label{thm-1}
	Let $\xx = x_1,\ldots,x_k$ denote an ordered sequence of elements of $R$. Let $M$ denote an $R$-module. 
	Then the following conditions are equivalent.
	\begin{itemize}
		\item[(i)] The sequence $\xx$ is $M$-proregular. 
		\item[(ii)] The sequence $\xx$ is $(M\otimes_RF)$-proregular for any flat $R$-module $F$.
		\item[(iii)] $\check{H}^1_{x_i}(\Gamma_{\xx_{i-1}}(\Hom_R(M,I) ) = 0$ for $i = 1,\ldots,k$ and any injective $R$-module $I$. 
		\item[(iv)] $\Gamma_{\xx_{i-1}}(\Hom_R(M,I))	/\Gamma_{\xx_i}(\Hom_R(M,I))$ is $x_i$-divisible for $i = 1,\ldots,k$ 
		and any injective $R$-module $I$.
		\end{itemize}
\end{theorem}

\begin{proof}
	The equivalence of (i) and (ii) holds trivially. Next we show the equivalence of (iii) and (iv). For an $R$-module $X$ and an element $y \in R$ there is an exact sequence 
	\[
	0 \to \Gamma_y(X) \to X \to X_y \to \check{H}^1_y(X) \to 0, 
	\]
	where $X_y$ denotes the localization of $X$ with respect to the element $y \in R$. Note that $X \to X_y$ is 
	just the \v{C}ech complex with respect to the single element $y \in R$. Next note that $\check{H}^1_y(X) = 0$ 
	if and only if $X \to X_y$ is onto or equivalently $X/\Gamma_y(X)$ is $y$-divisible. Applying this observation to 
	$\Gamma_{\xx_{i-1}}(\Hom_R(M,I))$ proves the equivalence of (iii) and (iv). 
	
	We use the abbreviations of the beginning of this section. Now we prove (i) $\Longrightarrow$ (iii).  Applying the functor $\Hom_R(\cdot,I)$ to the pro-zero 
	sequence of Koszul homologies (4) at the beginning of this section yields homomorphisms that provide 
	a direct system 
	\[
	H^1(y^n; H^0(\xx^{(n)};\Hom_R(M,I))) \stackrel{y^{m-n}}{\longrightarrow} H^1(y^m; H^0(\xx^{(m)};\Hom_R(M,I))).
	\eqno (5)
	\]
	If the system in (4) is pro-zero it follows that 
	\[
	0 = \varinjlim H^1(y^n; H^0(\xx^{(n)};\Hom_R(M,I)))  \cong \check{H} ^1_y(\Gamma_{\xx}( \Hom_R(M,I)). 
	\]
	The previous isomorphism might be checked directly, or see  \cite[6.1.11]{SpSa}.
	
	For the proof of  (iii) $\Longrightarrow$ (i) we note that $\xx^{(n)}M :_M y^n/\xx^{(n)}M \cong 
	\Hom_R(R/y^nR, M/\xx^{(n)}M)$.  Now fix $n$ and choose an injection 
	$f : (\xx^{(n)}M :_M y^n)/\xx^{(n)}M \hookrightarrow J$ into an injective $R$-module $J$. 
	This defines an element 
	\[
	f \in \Hom_R(	\Hom_R(R/y^nR, M/\xx^{(n)}M),J) \cong H^1(y^n; H^0(\xx^{(n)}; \Hom_R(M, J)).
	\]	
	By the assumption $\varinjlim H^1(y^n; H^0(\xx^{(n)}; \Hom_R(M, J)) \cong 
	\check{H}^1_y(\Gamma_{\xx}(\Hom_R(M,I) ) = 0$. Whence there is an integer $m \geq n$ such that 
	the image of $f$ in $\Hom_R(	\Hom_R(R/y^mR, M/\xx^{(m)}M),J)$ vanishes. In other words 
	the composite of the maps 
	\[
	 \xx^{(m)}M :_M y^m/\xx^{(m)}M  \stackrel{y^{m-n}}{\longrightarrow}  \xx^{(n)}M :_M y^n/\xx^{(n)}M  
	 \stackrel{f}{\longrightarrow} J
	\]
	is zero. This finishes the proof since $f$ is injective.
\end{proof}

As a first application we get  extensions of some of the authors' result in \cite{SS} to the case of an 
$R$-module $M$. 

\begin{corollary} \label{cor-1}
	Let $x \in R$ denote an element and let $M$ denote an $R$-module. Then the following 
	conditions are equivalent. 
	\begin{itemize}
		\item[(i)] $M$ is of bounded $x$-torsion. 
		\item[(ii) ]$M \otimes_RF$ is of bounded $x$-torsion for any flat $R$-module. 
		\item[(iii)] $\check{H}^1_x(\Hom_R(M,I)) = 0$ for any injective $R$-module.
		\item[(iv)] $\Hom_R(M,I)/\Gamma_x(\Hom_R(M,I) ) $ is $x$-divisible. 
	\end{itemize}
\end{corollary}

\begin{proof}
	This is a particular case of \ref{thm-1} for  $k = 1$. Namely $x$ is $M$-proregular 
	if and only if $M$ is of bounded $x$-torsion.
\end{proof}

For an $M$-regular sequence $\xx = x_1,\ldots,x_k$ it follows that $\xx^{(\underline{n}) } = 
x_1^{n_1}, \ldots,x_k^{n_k}$ is $M$-regular for any $k$-tupel $(n_1,\ldots,n_k) \in (\mathbb{N}_+)^k$. 
A corresponding result holds for $M$-proregular sequences. 

\begin{corollary} \label{cor-2}
	Let $\xx = x_1,\ldots,x_k$ be an ordered sequence of elements of $R$ and let $M$ be an $R$-module. 
	Then the following conditions are equivalent.
	\begin{itemize}
		\item[(i))] The sequence $\xx$ is $M$-proregular.
		\item[(ii)] There is an $\underline{n} \in (\mathbb{N}_+)^k$ such that the sequence 
		$\xx^{(\underline{n})}$ is $M$-proregular.
		\item[(iii)] The sequence $\xx^{(\underline{n})}$ is $M$-proregular for all $\underline{n} \in (\mathbb{N}_+)^k$.
	\end{itemize}
\end{corollary}

\begin{proof}
	The statements are easy consequences of Theorem \ref{thm-1}  by condition (iii). 
\end{proof}
Next we provide an alternative proof of \cite[A.2.3]{SpSa}  that an $M$-proregular sequence 
is also $M$-weakly prregular. This is done by the characterization of \ref{thm-1}

\begin{theorem} \label{thm-2}
	Let $\xx = x_1,\ldots,x_k$ be an ordered sequence in $R$. Let $M$ denote an $R$-module. If 
	$\xx$ is $M$-proregular, then $\xx$ is also $M$-weakly proregular. 
\end{theorem}

\begin{proof}
	We fix $l \in \{1,\ldots,k\}$ and put $\xx = x_1,\ldots,x_{l-1}$ and $y = x_l$. Then we show 
	by induction on $l$ that $\check{H}^i_{\xx, y} (\Hom_R(M,I) ) = 0$ for all $i >0$. If $l=k$  this proves the 
	claim by virtue of Proposition \ref{prop-3}.  If $l=1$ then $y$ is  $M$-proregular if and only if $M$ is of 
	bounded $y$-torsion. Whence the claim is true by \ref{cor-2}. Let $l >1$. Then there is the 
	short exact sequence 
	\[
	0 \to \check{H}^1_y(\check{H}^{i-1}_{\xx}(\Hom_R(M,I)) \to \check{H}^i_{\xx,y}(\Hom_R(M,I) ) \to 
	\check{H}^0_y(\check{H}^i_{\xx}(\Hom_R(M,I)) \to 0
	\]
	for all $i $ (see \cite[6.1.11]{SpSa}). By the induction hypothesis $\check{H}^i_{\xx}(\Hom_R(M,I)) =0$ for all $i >0$ and therefore
	$ \check{H}^i_{\xx,y}(\Hom_R(M,I) ) = 0$ for all $ i > 1$. The vanishing for $i =1$ holds since 
	$ \check{H}^1_y(\check{H}^0_{\xx}(\Hom_R(M,I)) = 0$ as it follows by the assumption (see Theorem  \ref{thm-1}).
	This completes the inductive step and finishes the proof. 
\end{proof}

\section{\v{C}ech (co-)complexes}
In order to continue with the study of weakly proregular sequences let us recall 
a few more definitions. 

\begin{notation} \label{not-2}
	(A) Let $\xx = x_1,\ldots,x_k$ denote a sytem of elements of $R$. As in \ref{not-1} we denote 
	by $\check{C}_{\xx}$ 
	the corresponding \v{C}ech complex. In \cite{SpSa} and \cite{Sp} there are constructions of two bounded complexes of free $R$-mdules $\check{L}_{\xx}$ and $\mathcal{L}_{\xx}$ and quasi-isomorphisms 
	$\check{L}_{\xx} \qism \mathcal{L}_{\xx} \qism \check{C}_{\xx}$. That is there are bounded free resolutions 
	of the  \v{C}ech complex  $\check{C}_{\xx}$  that is a complex of flat $R$-modules. \\
	(B) Let $M$ denote an $R$-module. With the previous notation the complex ${\RHom}_R( \check{C}_{\xx},M)$ 
	in the derived category has the following two representatives $\Hom_R(\mathcal{L}_{\xx} ,M) \qism 
    \Hom_R(\check{L}_{\xx} ,M)$ . \\
	(C) For an integer $i \in \mathbb{Z}$ we put 
	\[
	\check{H}_i^{\xx}(M) := H_i(\Hom_R(\mathcal{L}_{\xx} ,M) ) \cong 
	H_i( \Hom_R(\check{L}_{\xx} ,M))
	\]
	for the \v{C}ech homology of $M$ with respect to $\xx$. \\
	(D) For an $R$-module $M$ and an ideal $\mathfrak{a} \subset R$  we write $\Lambda^{\mathfrak{a}} (M) $ for the $\mathfrak{a}$-adic completion of $M$, i.e.  $ \Lambda^{\mathfrak{a}} (M) = \varprojlim M/\mathfrak{a}^n M$. The right derived functors of $\Lambda^{\mathfrak{a}} (M) $ are denoted by $\Lambda^{\mathfrak{a}}_i(M), i \in \mathbb{Z}$.  Note that in general $\Lambda^{\mathfrak{a}} (M) \not= \Lambda^{\mathfrak{a}}_0 (M) $ since $\Lambda^{\mathfrak{a}}$ is not left exact. 	
 \end{notation}

In the following we shall prove a relative version of one of the main results of \cite{SpSa}. We shall 
recover it for the case of $M = R$. 

\begin{theorem} \label{thm-3}
	Let $\xx = x_1,\ldots,x_k$ denote a sequence  of elements of $R$ and $\mathfrak{a} = \xx R$.  
	Let $M$ be an $R$-module.  Suppose that  $\xx$ is $M$-weakly proregular. 
	\begin{itemize}
		\item[(a)]   $\check{H}^i_{\xx}( \Hom_R(M,I)) = 0$ for all 
				$i > 0$ and  $\check{C}_{\xx} \otimes_R \Hom_R(M,I)$ is a right resolution of $\Gamma_{\mathfrak{a}}(\Hom_R(M,I))$ for any injective $R$-module $I$. 
				\item[(b)]  $	{\check{H}}_i^{\xx}(M \otimes_R F) = 0$  for all 	$i > 0$ and 
				$\Hom_R(\mathcal{L}_{\xx},M\otimes_R F)$ is a left resolution of $\Lambda^{\mathfrak{a}}(M\otimes_RF)$
				for any flat $R$-module $F$.
	\end{itemize}
\end{theorem}

\begin{proof}
	The vanishing result in (a) is shown in  \ref{prop-3}.  Moreover $\check{H}^0_{\xx}( \Hom_R(M,I)) 
	\cong \Gamma_{\mathfrak{a}}(\Hom_R(M,I)$. Now the natural morphism $\Gamma_{\mathfrak{a}}(\Hom_R(M,I) 
	\to \check{C}_{\xx} \otimes_R \Hom_R(M,I)$
   proves (a). 
    
     Now let us prove (b). By view of \cite[6.3.2 (b)]{SpSa} there are  the following short exact sequences 
    	\[
    	0 \to \varprojlim{}^1 H_{i+1}(\xx^{(n)}; M\otimes_R F)\to H_i({\RHom}_R( \check{C}_{\xx},M \otimes_RF) )
    	\to   \varprojlim H_i(\xx^{(n)}; M\otimes_R F) \to 0
    	\]
    	for all $i \in \mathbb{N}$. By the assumption $\{H_i(x^{(n)};M \otimes_R F)\}_{n \geq 1}$ is prozero for all 
    	$i > 0$. This yields that 
    	\[
    	\varprojlim{}^1 H_i(\xx^{(n)}; M\otimes_R F) = \varprojlim H_i(\xx^{(n)}; M\otimes_R F)  =0 
    	\]
    	for all $i >0$. This proves the the vanishing part.  For $i= 0$ we get the isomorphisms 
    	\[
    	H_i({\RHom}_R( \check{C}_{\xx},M \otimes_RF) ) \cong \varprojlim H_0(\xx^{(n)}; M\otimes_R F)  
    	\cong \Lambda^{\mathfrak{a}}(M\otimes_RF)
    	\]
    	as easily seen since $\xx R = \mathfrak{a}$. By \cite[8.2.1]{SpSa} and \cite[4.1 (B)]{Sp} 
    	there is a natural morphism $ \Hom_R(\mathcal{L}_{\xx},M\otimes_R F)\to\Lambda^{\mathfrak{a}}(M\otimes_RF)$ 
    	which completes the proof of (b).
\end{proof}

While the condition (a) in Theorem \ref{thm-3} is equivalent to the statement that $\xx$ is $M$-weakly 
proregular (see \ref{prop-3} ) this does not hold for (b). The following example is motivated by Lipman's 
example in \cite{lip2}.

\begin{example} \label{ex-2}
	Let $R = \Bbbk [|x|]$ denote the formal power series ring in the variable $x$ over the field 
	$\Bbbk$. Then define $S = \prod_{n \geq 1} R/x^nR$. By the component wise  operations $S$ becomes 
	a commutative ring. The natural map $R \to S, r \to (r+ x^n)_{n \geq 1}$, is a ring homomorphism and $x \mapsto {\bf x} := (x+x^n)_{n \geq 1}$.  As a direct product of $xR$-complete modules $S$ is an $xR$-complete $R$-module (see \cite[2.2.7]{SpSa}). 
	Since $R$ is a Noetherian ring $x$ is $R$-weakly proregular and $\check{H}^x_i(S) \cong H_i(\Hom_R(\mathcal{L}_x,S ))= 0$ for $i > 0$ and  $\check{H}^x_0(S) \cong H_0(\Hom_R(\mathcal{L}_x,S )) \cong S$. 
	
	Moreover, by the change of rings there is an isomorphism $\Hom_R(\mathcal{L}_x, S) \cong \Hom_S(\mathcal{L}_{{\bf x}}, S)$. That is, 
	$\check{H}_i^{{\bf x}} (S) = 0$ for $i > 0$ and $\check{H}_0^{{\bf x}}(S)\cong S$.  
	Now note that  $S$ is not of bounded ${\bf x}$-torsion as easily seen. It follows also that $S$ is $\bf x$-adic complete and $S$ is not a coherent ring.
\end{example}

In the following we shall apply the previous result to the situation of a complex of 
flat $R$-modules. It is a relative version of \cite[7.5.16]{SpSa}.

\begin{corollary} \label{cor-3}
	We fix the notation of \ref{thm-3}. Let $F_{\bullet}$
  	denote a complex of flat $R$-modules. Then the natural map 
	\[
	\Hom_R(\mathcal{L}_{\xx},M\otimes_R F_{\bullet}) \to \Lambda^{\mathfrak{a}}(M\otimes_RF_{\bullet})
	\]
	is a quasi-isomorphisms. Moreover, if $F_{\bullet}$ is a complex of finitely generated free $R$-modules, 
	then the natural map 
	\[
		\Hom_R(\mathcal{L}_{\xx},M\otimes_R F_{\bullet}) \to \Lambda^{\mathfrak{a}}(M)\otimes_R F_{\bullet}
	\]
	is a quasi-isomrphism.
\end{corollary}

\begin{proof}
	Let $F_{\bullet} = \{F_q, d_q\}_{q \in \mathbb{Z}}$ be the complex of flat $R$-modules . Then the natural map 
	 \[
	 	\Hom_R(\mathcal{L}_{\xx},M\otimes_R F_q) \to \Lambda^{\mathfrak{a}}(M \otimes_R F_q)
	 \]
	 is a quasi-isomorphism for all $q \in \mathbb{Z}$ (see \ref{thm-3}).   Since $	\Hom_R(\mathcal{L}_{\xx},M\otimes_R F_{\bullet})$ 
	 is the single complex associated to the double complex $\Hom_R(\mathcal{L}_{\xx}^p,M\otimes_R F_q)$ 
	 and because $\mathcal{L}_{\xx}$ is bounded the first claim follows (see also \cite[4.1.3]{SpSa}). 
	 For the second note that $F_q, q \in \mathbb{Z},$ is a finitely generated free $R$-module. Its tensor product commutes with inverse limits. 
\end{proof}

In the following we specify the previous result for a finitely generated $R$-module $N$. 

\begin{corollary} \label{cor-4}
	We use the assumption of \ref{thm-3}. Let $N$ denote a finitely generated $R$-module 
	with $L_{\bullet} \qism N$ a free resolution by finitely generated free $R$-modules. Then 
	there are isomorphisms 
	\[
	\check{H}_i^{\xx} (M\otimes_R L_{\bullet}) \cong \Tor_i^R(\Lambda^{\mathfrak{a}}(M), N)
	\]
	for all $i \in \mathbb{Z}$ and a spectral sequence 
	\[
	E^2_{i,j} = \check{H}^{\xx}_i(\Tor_j^R(M,N) ) \Longrightarrow E^{\infty}_{i+j}  = \Tor_{i+j}^R(\Lambda^{\mathfrak{a}}(M),N).	
	\]
\end{corollary}

\begin{proof}
	The isomorphisms are  an immediate consequence of \ref{cor-3}. The spectral sequence is just the 
	spectral sequence of the double complex (see e.g. \cite{Rj}).
\end{proof}

In the following we shall recall a particular case of one of the main results of \cite{SpSa}. 

\begin{corollary} \label{cor-5}
	Let $\xx = x_1,\ldots,x_k$ denote a sequence  of elements of $R$ and $\mathfrak{a} = \xx R$.  Suppose that $\xx$ is weakly proregular. Then $\check{H}_i^{\xx}(N) \cong \Lambda_i^{\mathfrak{a}}(N)$ for all $i \in \mathbb{N}$ 
	and any $R$-module $N$. 
\end{corollary}

\begin{proof} 
	Let $F_{\bullet} \qism N$ be a flat resolution of $N$. Then 
	\[
	\Hom_R(\mathcal{L}_{\xx},N) \qism 	\Hom_R(\mathcal{L}_{\xx}, F_{\bullet}) \qism \Lambda^{\mathfrak{a}}( F_{\bullet}).
	\]
	Since $\Lambda_i^{\mathfrak{a}}(N) = H_i(\Lambda^{\mathfrak{a}}( F_{\bullet}))$ the claim follows by taking 
	homology. 
\end{proof}

In fact the previous result holds for any complex $X$ of $R$-modules provided $\xx$ is weakly 
proregular (see \cite{SpSa} for these and more related information). Here we will continue 
with a further characterization of weakly proregular sequences. 

\begin{proposition} \label{prop-4}
	Let $\xx = x_1,\ldots,x_k$ denote a sequence of elelemts of $R$. Let $E$ 
	denote an injective cogenerator in the category of $R$-modules.  Then the following is equivalent.
	\begin{itemize}
		\item[(i)] The sequence $\xx$ is weakly proregular.
		\item[(ii)] $\check{H}^{\xx}_i(\Hom_R(I,J))= 0$ for all $i >0$ and any two injective $R$-module $I,J$. 
		\item[(iii)] $\check{H}^{\xx}_i(\Hom_R(I,E))= 0$ for all $i >0$ and any injective $R$-module $I$.
	\end{itemize}
\end{proposition}

\begin{proof}
		(i) $\Longrightarrow$ (ii): 
		By the adjointness there is the following isomorphisms of complexes 
		\[
		\Hom_R(\mathcal{L}_{\xx}, \Hom_R(I,J)) \cong \Hom_R(\mathcal{L}_{\xx} \otimes_RI,J)
. 		\]
		Therefore (i) implies that  $\Hom_R(H^i(\mathcal{L}_{\xx}  \otimes_R I),J) = 0$ for all $i > 0$. By the definition it yields (ii). \\
		(ii) $\Longrightarrow$ (iii): This holds trivially. \\
		(iii) $\Longrightarrow$ (i): By the previous adjointness isomorphism 
		it implies that $\Hom_R(H^i(\mathcal{L}_{\xx}  \otimes_R I),E) = 0$ for all 
		$i > 0$. Since 
		$E$ is an injective cogenerator this implies  $H^i(\mathcal{L}_{\xx}  \otimes_R I) \cong \check{H}^i_{\xx}(I)  = 0$ for all $i > 0$. Now this is equivalent to the fact that $\xx$ is weakly 
		proregular (see \ref{prop-3}).
\end{proof}

\begin{remark} (A)
	If $R$ is a coherent ring (see e.g. \cite[1.4.2]{SpSa}), then $\Hom_R(I,J)$ is a flat $R$-module for any two injective $R$-modules (see \cite[1.4.5]{SpSa}). \\
	(B) Let $R$ be a coherent ring. Then the conditions in \ref{prop-4} are 
	equivalent to 
	\begin{itemize}
		\item[(iv)] $\check{H}^{\xx}_i(F)= 0$  for all $i > 0$ and any flat $R$-module $F$.
	\end{itemize}
	This follows because $\Hom_R(I,E)$ is  $R$-flat for any injective $R$-module $I$ (see \cite[7.5.15]{SpSa} for more details). \\
	(C) There are examples of rings $R$ and injective $R$-modules $I,J$ 
	such that $\Hom_R(I,J)$ is not a flat $R$-module (see \cite[A.5.7]{SpSa}).
\end{remark}

\section{Local conditions}
Let $R$ denote a commutative ring. For an element $r \in R$ we write $D(f) = \Spec R \setminus V(f)$. Note that 
$D(f)$ is an open set in the Zariski topology of $\Spec R$. For $f 
\in R$ there is a natural map $\Spec R_f \to \Spec R$ that induces a homeomorphism between $\Spec R_f$ 
and $D(f)$. Since $\Spec R = \cup_{f \in R} D(f)$ and since $\Spec R$ is quasi-compact 
there are finitely many $f_1, \ldots, f_r \in R$ such that $\Spec R = \cup_{i=1}^r D(f_i)$. This provides the 
following definition.

\begin{definition} \label{def-3} (see \cite{St})
	A sequence $\ff = f_1,\ldots,f_r$ of elements of $R$ is called a covering sequence if $\Spec R = 
	\cup_{i=1}^r D(f_i)$. This is equivalent to saying that $R = \ff R$. Moreover, if $\ff$ is a covering sequence 
	then the natural map $M \to \oplus_{i=1}^r M_{f_i}$ is injective for any $R$-module $M$  as easily seen.
\end{definition}

Next we show proregularity as well as weakly proregularity are local properties. For a sequence $\xx = x_1,\ldots,x_k$ 
of elements of $R$ we denote by $\xx/1 = x/1, \ldots,x_k/1$ its image in a localization of $R$.

\begin{proposition} \label{prop-5}
	Let $\ff = f_1,\ldots,f_r$ denote a covering sequence of $R$. Let $\xx =x_1,\ldots,x_k$ be an ordered sequence of elements of $R$. For an $R$-module $M$  the following conditions are equivalent.
	\begin{itemize}
		\item[(i)] $\xx$  is $M$-weakly proregular (resp. proregular). 
		\item[(ii)] $\xx/1$ in $R_{f_i}$ is $M_{f_i}$-weakly proregular (resp. proregular) for all $i = 1,\ldots,r$. 
	\end{itemize}
\end{proposition}

\begin{proof}
	At first we consider the case of weakly proregular sequences. 
	Then for all integers $m\geq n$ and any integer $i > 0$ there is a commutative diagram 
	\[
	\begin{array}{ccc}
	H_i(\xx^{(m)};M) & \to & \oplus_{j=1}^r H_i(\xx/1^{(m)}; M_{f_j}) \\
	\downarrow& & \downarrow \\
	H_i(\xx^{(n)};M) & \to & \oplus_{j=1}^r H_i(\xx/1^{(n)}; M_{f_j})
	\end{array}
	\]
	where the horizontal maps are injective since $\ff$ is a covering sequence. If (i) holds the vertical maps at the left are zero for a given $n$ and an appropriate $m \geq n$. Since localization is flat and commutes 
	with homology the vertical maps at the right are zero too. For the converse fix $n$ and $i$ and choose 
	$m \geq n$ such that $H^i(\xx/1^{(m)}; M_{f_j}) \to H^i(\xx/1^{(n)}; M_{f_j})$ is zero for all $j = 1,\ldots,r$. 
	Then the vertical map at the right is zero.  Since the horizontal 
	maps are injective it follows that the vertical map at the left is zero. 
	
	The proof for the case of proregular sequence follows similar arguments by an inspection of the 
	natural map 
	\[
	(x_1^m,\ldots,x_{i-1}^m)M :_M x_i^m/(x_1^m,\ldots,x_{i-1}^m)M  \stackrel{x_i^{m-n}}{\longrightarrow} 
	(x_1^n,\ldots,x_{i-1}^n)M :_M x_i^n/(x_1^n,\ldots,x_{i-1}^n)M
	\]
	and the direct sum of the localizations with respect to $R_{f_i}, i = 1, \ldots, r$. 
\end{proof}

A corresponding local global principle is the following. 

\begin{proposition} \label{prop-6}
	 Let $\xx =x_1,\ldots,x_k$ be an ordered sequence of elements of $R$. For an $R$-module $M$  the following conditions are equivalent.
	\begin{itemize}
		\item[(i)] $\xx$  is $M$-weakly proregular (resp. proregular). 
		\item[(ii)] $\xx/1$ in $R_{\mathfrak{p}}$ is $M_{\mathfrak{p}}$-weakly proregular (resp. proregular) 
		for all $\mathfrak{p} \in \Spec R$. 
		\item[(iii)] $\xx/1$ in $R_{\mathfrak{m}}$ is $M_{\mathfrak{m}}$-weakly proregular (resp. proregular) 
		for all  maximal ideals $\mathfrak{m} \in \Spec R$. 
	\end{itemize}
\end{proposition}

\begin{proof}
	The proof follows easily by \cite[Theorem 4.6, p. 27]{mats}. We omit the details here. 
\end{proof}

By view of  \cite{St} we recall the following definition and extend it with the notion of proregularity. 

\begin{definition} \label{def-4} 
	(A)  (see \cite{St}) Let $R$ denote a commutative ring. An ideal $\mathcal{I} \subset R$ is 
	called an effective Cartier divisor 
	if there is a covering sequence $\ff = f_1,\ldots,f_r$ such that $\mathcal{I} R_{f_i}  = x_i R_{f_i}, i = 1,\ldots,r,$ 
	for non-zerodivisors $x_i$ of $R$. It follows that $\mathcal{I} \subseteq (x_1,\ldots,x_r)R$. \\
	(B) Let $\mathcal{I}$ denote a Cartier divisor and $x \in R $. The ideal $(\mathcal{I}, x)$ is 
	called proregular if 
	for any integer $n$ there is an integer $m\geq n$ such that $\mathcal{I}^m: x^m  \subseteq 
	\mathcal{I}^n : x^{m-n}$ . This is in consistence with the definition in \cite{GM} (see \ref{def-1}) and is equivalent 
	to the fact that for each $n$ there is an integer $m \geq n$ such that the multiplication map 
	$
	\mathcal{I}^m :_R x^m /\mathcal{I}^m \stackrel{x^{m-n}}{\longrightarrow} 
	\mathcal{I}^n :_R x^n /\mathcal{I}^n
	$
	is the zero map.
\end{definition}

In the following we shall consider a local global principle for proregular effective Cartier divisors. 

\begin{theorem} \label{prop-7}
	Let $\mathcal{I} \subseteq R$ an effective Cartier divisor with the covering sequence $\ff = f_1,\ldots, f_r$ 
	such that $\mathcal{I} R_{f_i}  = x_i R_{f_i}, i = 1,\ldots,r,$ 
	for non-zerodivisors $x_i$ of $R$. Suppose that $R/\mathcal{I} $ is of bounded $x$-torsion for some 
	$x \in R$. 
	\begin{itemize}
		\item[(a)]  $(\mathcal{I},x),$ is proregular, i.e. for each integer $n$ there is an  $m \geq n$ such that the multiplication map $$\mathcal{I}^m :_R x^m /\mathcal{I}^m \stackrel{x^{m-n}}{\longrightarrow} 
		\mathcal{I}^n :_R x^n /\mathcal{I}^n$$ is the zero map.
		\item[(b)]  $\Gamma_{\mathcal{I}} (I)/\Gamma_{(\mathcal{I},x)}(I)$ is $x$-divisible for any 
		injective $R$-module $I$.
	\end{itemize}
Moreover, the two conditions (a) and (b) are equivalent. 
\end{theorem}

\begin{proof}
	At first we prove (a). Since $R/\mathcal{I}$ is of bounded $x$-torsion there is an integer $c$ such that 
	$\mathcal{I} :_R x^m = \mathcal{I}:_R x^c$  for all $m \geq c$. By localization at $R_{f_i}$ it follows that 
	$x_i R_{f_i} :_{R_{f_i} }  x/1^m = x_i R_{f_i} :_{R_{f_i} }  x/1^c$ for all $m \geq c$ and $i = 1,\ldots,r$ with $R_{f_i}$-regular elements $x_i/1 \in R_{f_i}$.  By view of  \ref{prop-1} it follows $x_i/1, x/1$ is 
	an $R_{f_i}$-proregular sequence for all $i = 1,\ldots,r$. For a given $n$ and $m \geq n$ there is the 
	following commutative diagram 
		\[
	\begin{array}{ccc}
	\mathcal{I}^m :_R x^m/\mathcal{I}^m& \to & \oplus_{j=1}^r  (x_i ^mR_{f_i} :_{R_{f_i} }  x/1^m)/x_i^mR_{f_i}\\
	\downarrow^{x^{m-n}}& & \downarrow ^{\oplus (x^{m-n}/1)}\\
	\mathcal{I}^n :_R x^n/\mathcal{I}^n& \to & \oplus_{j=1}^r  (x_i ^nR_{f_i} :_{R_{f_i} }  x/1^n)/x_i^nR_{f_i}
	\end{array}
	\]
	Now choose $m$ such that multiplication at the vertical maps at the right are all zero. Then the multiplication 
	at the left is the zero map too since the horizontal maps are injective as follows by the localization (see \ref{def-3}). 
	
	For the proof of (b) recall that the natural map $\mathcal{I}^m :_R x^m /\mathcal{I}^m \stackrel{x^{m-n}}{\longrightarrow} 
	\mathcal{I}^n :_R x^n /\mathcal{I}^n$ coincides with the map induced by the Koszul complexes, i.e.,
	\[
	H_1(x^m; R/\mathcal{I}^m) \stackrel{x^{m-n}}{\longrightarrow}  	H_1(x^n; R/\mathcal{I}^n) .
	\]
	Now apply $\Hom_R(\cdot,I)$ with an arbitrary injective $R$-module. Then the induced map 
	\[
	H^1(x^n; \Hom_R(R/\mathcal{I}^n, I))  \to H^1(x^m; \Hom_R(R/\mathcal{I}^m, I)) 
	\]
	is zero and $0 = \varinjlim H^1(x^m; \Hom_R(R/\mathcal{I}^m, I))  \cong \check{H}^1_x(\Gamma_{\mathcal{I}}(I))$ 
	which proves the statement  as in the proof of \ref{thm-1}. For the implication (b) $\Longrightarrow$ (a) we fix $n$ and choose an injection $$H^1(x^n; \Hom_R(R/\mathcal{I}^n, I)) \cong 	\mathcal{I}^n :_R x^n /\mathcal{I}^n 
	\hookrightarrow I$$ into an injective $R$-module $I$. By the vanishing of the direct limit 
	there is,  as in the proof of Theorem \ref{thm-1}, an integer $m \geq n$ such that (a) holds. 
\end{proof}

In the following we shall give a comment of the previous investigations to the recent work 
of Bhatt and Scholze (see \cite{BSc}). To this end let $p \in \mathbb{N}$ denote a prime number 
and let $\mathbb{Z}_p := \mathbb{Z}_{\mathfrak{p}}$ the localization at  the prime ideal 
$(p)= \mathfrak{p}\in \Spec \mathbb{Z}$. In the following let $R$ be a $\mathbb{Z}_p$-algebra. 

\begin{definition} \label{def-5} (see \cite[Definition 1.1]{BSc})
	A prism is a pair $(R,\mathcal{I}$) consisting of a $\delta$-ring $R$ (see \cite[Remark 1.2]{BSc}) and 
	a Cartier divisor $\mathcal{I}$ on $R$ satisfying the following two conditions.
	\begin{itemize}
		\item[(a)] The ring $R$ is $(p,\mathcal{I})$-adic complete.
		\item[(b)] $p \in \mathcal{I} + \phi_R(\mathcal{I})R$, where $\phi_R$ is the lift of the Frobenius on $R$ 
		induced by its $\delta$-structure (see \cite[Remark 1.2]{BSc}).
	\end{itemize}
\end{definition}

With the previous definition there is the following application of our results.

\begin{corollary} \label{cor-6}
	Let $(R, \mathcal{I})$ denote a prism. Suppose that $\mathcal{I}$ is of bounded $p$-torsion. Then 
	\begin{itemize}
		\item[(a)] 	$(\mathcal{I}, p)$ is proregular in the sense of \ref{def-4}. 
		\item[(b)]  $\Gamma_{\mathcal{I}} (I)/\Gamma_{(\mathcal{I},p)}(I)$ is $p$-divisible for any 
		injective $R$-module $I$, i.e. $\Gamma_{\mathcal{I}} (I) = p \Gamma_{\mathcal{I}} (I) + 
		\Gamma_{(\mathcal{I},p)}(I)$.
	\end{itemize}
The conditions in (a) and (b) are equivalent.
\end{corollary}

\begin{proof}
	This is an immediate consequence of \ref{prop-7}.
\end{proof}

We note that Yekutieli (see  \cite[Theorem 7.3]{Y}) has slightly modified the notion of weakly proregularity and has shown that $(\mathcal{I}, p)$ is weakly proregular  under the assumption of \ref{cor-6}. It should be mentioned that proregularity is more strong than 
weakly proregularity as shown by the example in \cite{lip2}. 

While the notion of weakly 
proregularity plays an essential role in the study of local (co-) homology (see \cite{SpSa} and the references 
there) the previous statements seem to be a further application of the notion of proregularity as introduced 
by Greenlees and May (see  \cite{GM}) and by Lipman (see \cite{lip}). 

\medskip 
{\sc Acknowledgement}. The author thanks Anne-Marie Simon (Universit\'{e} Libre de Bruxelles) for various discussions about the subject 
and a careful reading of the manuscript.

\end{document}